\documentclass[12pt]{amsart}
\usepackage{amsfonts,amsthm,latexsym,amsmath,amssymb,amscd,amsmath, mathrsfs, url, cite, filecontents, mathtools}

\newcounter{count}

\numberwithin{count}{section}
\newtheorem{Lemma}[count]{Lemma}

\newtheorem{Theorem}[count]{Theorem}

\begin{document}

\author[T.~H.~Nguyen]{Thu Hien Nguyen}
\address{Department of Mathematics \& Computer Sciences, 
V.~N.~Karazin Kharkiv National University,
4 Svobody Sq., Kharkiv, 61022, Ukraine}
\email{nguyen.hisha@karazin.ua}

\author[A.~Vishnyakova]{Anna Vishnyakova}
\address{Department of Mathematics \& Computer Sciences, V.~N.~Karazin Kharkiv National University,
4 Svobody Sq., Kharkiv, 61022, Ukraine}
\email{anna.m.vishnyakova@univer.kharkov.ua}

\title[Entire functions of the Laguerre--P\'olya class]
{On the entire functions from the Laguerre--P\'olya I class 
having the increasing  second quotients of Taylor coefficients }

\begin{abstract}

We prove that if $f(x) = \sum_{k=0}^\infty  a_k x^k,$
$a_k >0,  $ is an entire function 
such that the sequence $Q := \left( \frac{a_k^2}{a_{k-1}a_{k+1}}  \right)_{k=1}^\infty$ 
is non-decreasing and $\frac{a_1^2}{a_{0}a_{2}} \geq 2\sqrt[3]{2},$
then all but a finite number of zeros of $f$ are real and
simple. We also present a  criterion in terms of the closest to zero roots for such 
a function to have only real zeros (in other words, for belonging to the Laguerre--P\'olya 
class of type I)  under additional assumption on the sequence $Q.$ 

\end{abstract}



\maketitle

\section{Introduction}

In this paper, we study conditions under which special entire functions with
positive Taylor coefficients have only real zeros. 
First, we need the definition of two well-known classes of entire functions: the 
Laguerre--P\'olya class and the Laguerre--P\'olya class of type I.

{\bf Definition 1}. {\it A real entire function $f$ is said to be in the {\it
Laguerre--P\'olya class}, written $f \in \mathcal{L-P}$, if it can
be expressed in the form
\begin{equation}
\label{lpc}
 f(x) = c x^n e^{-\alpha x^2+\beta x}\prod_{k=1}^\infty
\left(1-\frac {x}{x_k} \right)e^{xx_k^{-1}},
\end{equation}
where $c, \alpha, \beta, x_k \in  \mathbb{R}$, $x_k\ne 0$,  $\alpha \ge 0$,
$n$ is a nonnegative integer and $\sum_{k=1}^\infty x_k^{-2} < \infty$. 

A real entire function $f$ is said to be in the {\it Laguerre--
P\'olya class of type I}, 
written $f \in \mathcal{L-P} I$, if it can
be expressed in the following form  
\begin{equation}  \label{lpc1}
 f(x) = c x^n e^{\beta x}\prod_{k=1}^\infty
\left(1+\frac {x}{x_k} \right),
\end{equation}
where $c \in  \mathbb{R},  \beta \geq 0, x_k >0 $, 
$n$ is a nonnegative integer,  and $\sum_{k=1}^\infty x_k^{-1} <
\infty$. 
As usual, the product on the right-hand side can be
finite or empty (in the latter case the product equals 1).}

These classes are important for the theory of entire functions since 
the polynomials with only real zeros (only real and nonnegative zeros) converge locally 
uniformly to these and only these functions. 
The following  prominent theorem 
states an even  stronger fact. 

{\bf Theorem A} (E.~Laguerre and G.~P\'{o}lya, see, for example,
\cite[p. ~42--46]{HW}) and \cite[chapter VIII, \S 3]{lev}). {\it   

(i) Let $(P_n)_{n=1}^{\infty},\  P_n(0)=1, $ be a sequence
of complex polynomials having only real zeros which  converges uniformly on the disc 
$|z|\leq A, A > 0.$ Then this sequence converges locally uniformly to an entire function 
from the $\mathcal{L-P}$ class.

(ii) For any $f \in \mathcal{L-P}$ there exists a sequence of complex polynomials 
with only real zeros which converges locally uniformly to $f$.

(iii) Let $(P_n)_{n=1}^{\infty},\  P_n(0)=1, $ be a sequence
of complex polynomials having only real negative zeros which  
converges uniformly on the disc $|z| \leq A, A > 0.$ Then this 
sequence converges locally uniformly to an entire function
 from the class $\mathcal{L-P}I.$
 
(iv) For any $f \in \mathcal{L-P}I$ there is a
sequence of complex polynomials with only real nonpositive 
zeros which  converges locally uniformly to $f$.}

For  interesting properties and characterizations of the
Laguerre--P\'olya   class and the Laguerre--P\'olya class of type I,  see 
\cite[p. 100]{pol}, \cite{polsch}  or  \cite[Kapitel II]{O} (also see the
survey \cite{iv} on the zero distribution of  entire functions, its sections and tails).
Note that for a real entire function (not identically zero) of the order less than $2$ having 
only real zeros is  equivalent to belonging to the Laguerre--P\'olya class. Also, 
for a real entire function with positive coefficients of the order 
less than $1$ having only real zeros is  equivalent to belonging to the Laguerre--P\'olya 
class of type I.

Let  $f(z) = \sum_{k=0}^\infty a_k z^k$  be an entire function with 
real nonzero coefficients. We define the quotients $p_n$ and $q_n$:
\begin{align*}
p_n=p_n(f)&:=\frac{a_{n-1}}{a_n}, \quad n\geq 1, \\
q_n=q_n(f)&:= \frac{p_n}{p_{n-1}} = \frac{a_{n-1}^2}{a_{n-2}a_n}, \quad n\geq 2.
\end{align*}

From these definitions it readily follows that
\begin{align*}
& a_n=\frac{a_0}{p_1p_2\cdots p_n}, \quad n\geq 1, \\
& a_n = a_1\Big(\frac{a_1}{a_0} \Big)^{n-1} \frac{1}{q_2^{n-1}q_3^{n-2}
\cdot \ldots \cdot q_3^2 q_2}, \quad n\geq 2.
\end{align*}

It is not trivial to understand whether a given entire function has only real zeros. 
In 1926, J. I. Hutchinson found the following 
sufficient condition for an entire function with positive coefficients to 
have only real zeros.

{\bf Theorem B} (J. ~I. ~Hutchinson, \cite{hut}). { \it Let $f(z)=
\sum_{k=0}^\infty a_k z^k$, $a_k > 0$ for all $k$. 
Then $q_n(f)\geq 4$, for all $n\geq 2,$  
if and only if the following two conditions are fulfilled:\\
(i) The zeros of $f(z)$ are all real, simple and negative, and \\
(ii) the zeros of any polynomial $\sum_{k=m}^n a_kz^k$, $m < n,$  formed 
by taking any number 
of consecutive terms of $f(z) $, are all real and non-positive.}

For some extensions of Hutchinson's results see,
for example, \cite[\S4]{cc1}.

The entire function $g_a(z) =\sum _{k=0}^{\infty} z^k a^{-k^2}$, $a>1,$
known as a \textit{partial theta-function},  was investigated in many works and has a 
prominent role. We can also observe that $q_n(g_a)=a^2$ for all $n.$ The  
survey \cite{War} by S.~O.~Warnaar contains the history of
investigation of the partial theta-function and some of its main properties.

In \cite{klv} it is shown that  for every $n\geq 2$ there exists a constant $c_n >1$ 
such that  $S_{n}(z,g_a):=\sum _{j=0}^{n} z^j a^{-j^2} \in \mathcal{L-P}$ 
$ \Leftrightarrow \ a^2 \geq c_n$.

{\bf Theorem C} (O. ~Katkova, T. ~Lobova, A. ~Vishnyakova, \cite{klv}).  {\it There exists a constant 
$q_\infty $ $(q_\infty\approx 3{.}23363666 \ldots ) $ such that:
\begin{enumerate}
\item
$g_a(z) \in \mathcal{L-P} \Leftrightarrow \ a^2\geq q_\infty ;$
\item
$g_a(z) \in \mathcal{L-P} \Leftrightarrow \ $  there exists $x_0 \in (- a^3, -a)$ 
such that $ \  g_a(x_0) \leq 0;$
\item
for a given $n\geq 2$ we have $g_{a,n}(z) \in \mathcal{L-P}$ $ \ \Leftrightarrow \ $
there exists $x_n \in (- a^3, -a)$ such that $ \ g_{a,n}(z) \leq 0;$
\item
$ 4 = c_2 > c_4 > c_6 > \cdots $  and    $\lim_{n\to\infty} c_{2n} = q_\infty ;$
\item
$ 3= c_3 < c_5 < c_7 < \cdots $  and    $\lim_{n\to\infty} c_{2n+1} = q_\infty .$
\end{enumerate}}

There is a series of works by V.~P. ~Kostov dedicated to the interesting properties of zeros of 
the partial theta-function and its  derivative (see \cite{kos0, kos1, kos2, kos3, kos03, kos04, kos4, kos5, kos5.0}). 
A wonderful paper \cite{kosshap} among the other results explains the role 
of the constant $q_\infty $  in the study of the set of entire functions with positive 
coefficients having all Taylor truncations with only real zeros. 

A.~D.~Sokal in \cite{sokal} studies the leading roots of the partial theta-function. A formal power series 
\begin{align*}
f(x,y) = \sum_{n=0}^{\infty} \alpha_n x^ny^{n(n-2)/2},
\end{align*}
(where the coefficients $(\alpha_n)_{n=0}^\infty$ belong to a commutative ring with identity element,  
and $\alpha_0 = \alpha_1 = 1$) is considered as a formal power series in $y$ whose coefficients are 
polynomials in $x$. A.~D.~Sokal defines the "leading root" of $f$ as a unique formal power series $x_0(y) \in R[y]$ 
which satisfies the equation $f(x_0(y),y) = 0.$ The coefficientwise positivity of $-x_0(y)$ was proved. Moreover, 
all the coefficients of $1/x_0(y)$ and $1/x_0(y)^2$ after the constant term 1 are strictly negative, except 
for the vanishing coefficient of $y^3$ for the latter case.

In \cite{klv1}, the following questions are investigated: whether the Taylor sections 
of the function $\prod \limits_{k=1}^\infty \left(1 + \frac{z}{a^k} \right)$, where 
$a > 1,$ and $\sum_{k=0}^{\infty}\frac{z^k}{k!a^{k^2}}$ belong to the 
Laguerre--P\'olya class of type I. In \cite{BohVish} and 
\cite{Boh}, some important special functions with increasing sequence 
of the second quotients 
of Taylor coefficients  are studied. 

B.~He in \cite{he} considers the entire function as follows
\begin{align*}
A_q^{(\alpha)}(a;z) = \sum_{k=0}^{\infty} \frac{(a;q)_kq^{\alpha k^2}z^k}{(q;q)_k},
\end{align*}

where $\alpha >0, 0 < q < 1$ and 

\begin{align*}
(a;q)_n = 
	\begin{cases}
	1, n = 0\\
	\prod\limits_{j = 1}^{n-1} (1 - aq^j) (n \geq 1)
	\end{cases}
\end{align*}
is the q-shifted factorial.
The entire function $A_q^{(\alpha)}(a;z)$ defined as above is the generalisation of Ramanujan entire
 function and the Stieltjes-Wigert polynomial which have only real positive zeros. The paper gives a 
 partial answer to Zhang's question: under what conditions the zeros of the entire function 
 $A_q^{(\alpha)}(a;z)$ are all real.

In \cite{ngthv1} it is proved that if $f(z)=\sum_{k=0}^\infty a_k z^k $, $a_k > 0$ 
for all $k,$ is an entire function such that  $q_2 \geq q_3 \geq 
q_4 \geq \cdots, $  and  $\lim\limits_{n\to \infty} q_n(f) = b \geq q_\infty, $
then  $f \in \mathcal{L-P}$.

In \cite{ngthv2} it is proved that if $f(z)=\sum_{k=0}^\infty a_k z^k $, $a_k > 0$ 
for all $k,$ is an entire function such that $q_2 \leq q_3 \leq 
q_4 \leq \cdots, $   and  $\lim\limits_{n\to \infty} q_n(f) = c < q_\infty,$
then the function $f$ does not belong to the  Laguerre--P\'olya class.

The first author studied a special function related to the partial theta-function and 
the Euler function, $f_a(z) = \sum_{k=0}^\infty \frac{z^k}{(a^k+1)(a^{k-1}+1)
\cdot \ldots \cdot (a+1)},$ $a>1,$  which is also known 
as the $q$-Kummer function $\prescript{}{1}{\mathbf{\phi}}_1(q;-q; q,-z)$, 
where $q=1/a$ (see \cite{GR}, formula (1.2.22)), and found the conditions for it to 
belong to the Laguerre--P\'olya class (see \cite{ngth1}).

A.~Baricz and S.~Singh in \cite{BarSingh} investigated the Bessel functions. The Hurwitz 
theorem on the exact number of nonreal zeros was extended for the Bessel functions of the 
first kind. In addition, the results on zeros of derivatives of Bessel functions and the cross-product 
of Bessel functions were obtained.

It turns out that for many important entire functions with positive
coefficients $f(z)=\sum_{k=0}^\infty a_k z^k $ (for example, partial theta-function
from \cite{klv}, functions from \cite{BohVish} and \cite{Boh}, 
the $q$-Kummer function $\prescript{}{1}{\mathbf{\phi}}_1(q;-q; q,-z)$ 
$= \sum_{k=0}^\infty \frac{z^k}{(a^k+1)(a^{k-1}+1)
\cdot \ldots \cdot (a+1)},$ $a>1,$ and others) the following two conditions
are equivalent: 

(i) $f$ belongs to the Laguerre--P\'olya class of type I,  

and 

(ii) there exists $z_0 \in [-\frac{a_1}{a_2},0]$ such that $f(z_0) \leq 0.$

In our previous work, we have obtained a new necessary condition for an entire function to belong to the 
Laguerre--P\'olya class of type I in terms of the closest to zero roots.

{\bf Theorem D} (T. ~H. ~Nguyen, A. ~Vishnyakova, \cite{ngthv3}).
{\it Let $f(z)=\sum_{k=0}^\infty a_k z^k $, $a_k > 0$ for all $k,$  be an 
entire function.  Suppose that the quotients $q_n(f)$ satisfy the following condition: 
$q_2(f) \leq q_3(f).$ If the function $f$ belongs to the  Laguerre--P\'olya class, then 
there exists  $z_0 \in [-\frac{a_1}{a_2},0]$ such that $f(z_0) \leq 0$.} 

Our first result concerns the possible number of nonreal zeros 
of a real entire function whose sequence of  the second quotients 
of Taylor coefficients is  non-decreasing.

\begin{Theorem}
\label{th:mthm2}
Let $f(x) = \sum_{k=0}^\infty  a_k x^k,$
$a_k >0,  k=0, 1, 2, \ldots ,$ be an entire function 
such that $2\sqrt[3]{2} \leq q_2(f) \leq q_3(f) \leq q_4(f) \leq \ldots$
Then all but a finite number of zeros of $f$ are real and
simple.
\end{Theorem}

In connection with the theorem above, we formulate the following conjecture.

{\bf Conjecture.}  {\it  Let $f(x) = \sum_{k=0}^\infty  a_k x^k,$
$a_k >0,  k=0, 1, 2, \ldots ,$ be an entire function 
such that $1 < q_2(f) \leq q_3(f) \leq q_4(f) \leq \ldots$
Then all but a finite number of zeros of $f$ are real and
simple.}

The second result of our paper is the following  criterion for belonging to the 
Laguerre--P\'olya class of type I for real entire
functions with the regularly non-decreasing sequence of second quotients 
of Taylor coefficients.  To clarify the statement of the next theorem we need the following 
lemma from \cite{ngthv3}.

\begin{Lemma}(Lemma 1.2 from \cite{ngthv3}, cf. Lemma 2.1 from \cite{ngthv2}).
\label{lem0}
If $f(z) = \sum_{k=0}^\infty a_kz^k, a_k>0,$ belongs to $\mathcal{L-P}I,$ 
then $q_3(q_2 - 4) + 3 \geq 0.$
In particular, if $q_3 \geq q_2,$ then $q_2 \geq 3.$
\end{Lemma}

So, if we investigate whether a real entire function, with the non-decreasing 
sequence of second quotients of Taylor coefficients, belongs to the 
Laguerre--P\'olya class of type I, then the necessary condition is $q_2 \geq 3.$
Our main result is the following theorem.

\begin{Theorem}
\label{th:mthm1}
Let $f(z)=\sum_{k=0}^\infty a_k z^k$, $a_k > 0$ for all $k,$  be an 
entire function. Suppose that  
$3 \leq q_2(f) \leq q_3(f) \leq q_4(f) \leq \ldots$. Suppose also that
if there is an integer $j_0 \geq 2,$ such that $q_{j_0}(f) < 4,$ and  
$q_{j_0+1}(f) \geq 4,$  then $\frac{q_{j_0 -1}(f)}{q_{j_0+1}(f)} \geq 0{.}525,$ 
or $q_{j_0}(f) \geq 3{.}4303. $ 
Then $f \in \mathcal{L-P}I$ if 
and only if there exists $z_0 \in [-a_1/a_2,0]$ such that $f(z_0) \leq 0.$
\end{Theorem}

Unfortunately, at the moment we do not know whether the additional
assumptions in the above theorem  are essential. 

\section{Proof of Theorem \ref{th:mthm2}}

\begin{proof} To prove Theorem \ref{th:mthm2} we need the
following Lemma.

\begin{Lemma}
\label{th:lm3}
Let $f(x) = \sum_{k=0}^\infty (-1)^k a_k x^k,$
$a_k >0,  k=0, 1, 2, \ldots ,$ be an entire function 
such that $2\sqrt[3]{2} \leq q_2(f) \leq q_3(f) \leq q_3(f) \leq \ldots$ 
For an arbitrary integer $j \geq 2$ we define 
$$\rho_j(f) := q_2(f)q_3(f) \cdots  q_j(f) \sqrt{q_{j+1}(f)}.$$ 
Then, for all sufficiently large $j$,  the function $f$ has 
exactly $j$ zeros on the disk $\{z\  :\   |z| <  \rho_j(f) \}$
counting multiplicities.
\end{Lemma}

\begin{proof}
For simplicity, we will write $q_j$ instead of $q_j(f)$ and  $\rho_j$
instead of $\rho_j(f).$  We have
\begin{align*}
f(z) = \sum_{k=0}^\infty \frac{(-1)^k z^k}{q_2^{k-1}q_3^{k-2} \cdots q_k},
\end{align*}
where the sequence $q_2, q_3, \ldots$ is non-decreasing.
We now dissect the above sum as
\begin{align*}
\sum_{k=0}^\infty \frac{(-1)^k z^k}{q_2^{k-1}q_3^{k-2} \cdots q_k} = 
\bigg(\sum_{k=0}^{j-3} + \sum_{k = j-2}^{j+2} + 
\sum_{k = j+3}^{\infty} \bigg)=: \Sigma_{1,j}(z) + g_j(z) + \Sigma_{2,j}(z),
\end{align*}
where
\begin{eqnarray}
\label{ll1}
& g_j(z) = \left( \sum_{k = j - 2}^{j+1} 
\frac{(-1)^kz^k}{q_2^{k-1} q_3^{k-2} \cdots q_k} 
+ \frac{(-1)^{j+2}z^{j+2}}{q_2^{j+1}q_3^j 
\cdots q_{j-2}^5q_{j-1}^4q_j^4q_{j+1}^2}\right) \\
\nonumber &
+\frac{(-1)^{j+2}z^{j+2}}{q_2^{j-3}q_3^{j-4} 
\cdots q_{j-2}} \bigg( \frac{1}{q_2^4q_3^4 
\cdots q_{j-1}^4q_j^3q_{j+1}^2q_{j+2}} - 
\frac{1}{q_2^4q_3^4 \cdots q_{j-1}^4q_j^4q_{j+1}^2}\bigg) \\
\nonumber & =: \widetilde{g}_j(z) + \xi_j(z).
\end{eqnarray}

By the definition of $\rho_j$ we have $q_2q_3 \cdots q_j  < 
\rho_j < q_2q_3 \cdots q_jq_{j+1}$.
We get
\begin{eqnarray}
\label{ll2}
&
(-1)^{j-2}g_j(\rho_je^{i\theta}) = 
e^{i(j-2)\theta}q_2q_3^2 \cdots q_{j-2}^{j-3}q_{j-1}^{j-2}q_j^{j-2}
q_{j+1}^{\frac{j-2}{2}}  \cdot \\ \nonumber &
 \bigg( 1 - e^{i\theta}q_j \sqrt{q_{j+1}}
+ e^{2i\theta}q_jq_{j+1} - e^{3i\theta}q_j\sqrt{q_{j+1}} + 
e^{4i\theta}q_jq_{j+2}^{-1}\bigg) \\  \nonumber &
= e^{i(j-2)\theta}q_2q_3^2 \cdots 
q_{j-2}^{j-3}q_{j-1}^{j-2}q_j^{j-2}q_{j+1}^{\frac{j-2}{2}} \cdot  \\
\nonumber &  \bigg( 1 - e^{i\theta}q_j \sqrt{q_{j+1}}
+ e^{2i\theta}q_jq_{j+1} - e^{3i\theta}q_j\sqrt{q_{j+1}} + 
e^{4i\theta}\bigg) \\   \nonumber &
+ e^{i(j+2)\theta}q_2q_3^2 \cdots 
q_{j-2}^{j-3}q_{j-1}^{j-2}q_j^{j-2}q_{j+1}^{\frac{j-2}{2}} 
\bigg(q_jq_{j+2}^{-1} - 1\bigg) \\  \nonumber &
= \widetilde{g}_j(\rho_je^{i\theta}) + \xi_j(\rho_je^{i\theta}).
\end{eqnarray}

Our aim is to show that for every sufficiently large $j$ the following inequality holds:  
\begin{align*}
\min_{0\leq \theta\leq 2\pi}|\widetilde{g}_j(\rho_j e^{i\theta})| > 
\max_{0\leq \theta\leq 2\pi}|f
(\rho_j e^{i\theta}) -\widetilde{g}_j(\rho_j e^{i\theta})|, 
\end{align*} 
so that the number of zeros of  $f$  in the circle $\{z : |z| < \rho_j  \}$  
is equal to  the number of zeros of $\widetilde{g}_j$ in the same circle. 
Later in the proof, we also find the number of zeros. 
First, we find $\min_{0 \leq \theta \leq 2\pi}  |\widetilde{g}_j(\rho_je^{i\theta})|.$
\begin{eqnarray}
\label{ll3}
&
\widetilde{g}_j(\rho_je^{i\theta}) 
=e^{ij\theta}q_2q_3^2 \cdots q_{j-2}^{j-3}q_{j-1}^{j-2}
q_j^{j-2}q_{j+1}^{\frac{j-2}{2}} \cdot \\  \nonumber &
 \bigg( e^{-2i\theta} - e^{-i\theta}q_j \sqrt{q_{j+1}} + q_jq_{j+1} -
e^{i\theta}q_j \sqrt{q_{j+1}}  + e^{2i\theta}\bigg)\\
 \nonumber & = e^{ij\theta}q_2q_3^2 \cdots 
q_{j-2}^{j-3}q_{j-1}^{j-2}q_j^{j-2}q_{j+1}^{\frac{j-2}{2}} \cdot \\
 \nonumber & \bigg( 2\cos 2\theta - 
2\cos\theta q_j \sqrt{q_{j+1}}
+ q_jq_{j+1} \bigg) \\
 \nonumber & =: e^{ij\theta}q_2q_3^2 \cdots q_{j-2}^{j-3}
q_{j-1}^{j-2}q_j^{j-2}q_{j+1}^{\frac{j-2}{2}} 
\cdot  \psi_j(\theta).
\end{eqnarray}

We consider $\psi_j(\theta)$ as following
\begin{align*}
\psi_j(\theta) =\widetilde{\psi}_j(t) := 4t^2 - 2q_j
\sqrt{q_{j+1}}t + (q_jq_{j+1} - 2),
\end{align*}
where $t:= \cos \theta,$ and where we have used 
that $\cos 2\theta = 2t^2 - 1.$

The vertex of the parabola is $t_j = q_j\sqrt{q_{j+1}}/4.$ Under our assumptions, 
$2\sqrt[3]{2} \leq q_2 \leq q_3 \leq \ldots,$  so that $q_j\sqrt{q_{j+1}}/4 \geq q_2\sqrt{q_{2}}/4
\geq 2\sqrt[3]{2} \sqrt{2\sqrt[3]{2}}/4 =1, $ and we have $t_j \geq 1.$
Hence, 
$$\min_{t \in [-1, 1]} \widetilde{\psi}_j(t) = \widetilde{\psi}_j(1) =
2 - 2q_j\sqrt{q_{j+1}} + q_j q_{j+1} = q_j\sqrt{q_{j+1}}(\sqrt{q_{j+1}} -2)  + 2.$$
If  $q_{j+1} \geq 4,$ then $q_j\sqrt{q_{j+1}}(\sqrt{q_{j+1}} -2)  + 2>0. $
If $q_{j+1} <4$, then $q_j\sqrt{q_{j+1}}(\sqrt{q_{j+1}} -2)  + 2 \geq 
q_{j+1}\sqrt{q_{j+1}}(\sqrt{q_{j+1}} -2)  + 2 =q_{j+1}^2 - 2 q_{j+1}
\sqrt{q_{j+1}} +2. $ Denote by $y = \sqrt{q_{j+1}} \geq 0,$ and
$g(y) = y^4 - 2 y^3 +2.$ It is easy to calculate that $\min_{y \geq 0} g(y) =
g(\frac{3}{2}) = \frac{5}{16} >0.$
Therefore,  we get
$$2 - 2q_j\sqrt{q_{j+1}} + q_j q_{j+1}> 0.$$
Thus, $\widetilde{\psi}_j(t) > 0$ for all $t \in [-1, 1]$.
Consequently, we have obtained the estimate from below:
\begin{eqnarray}
\label{estg}
& \min_{0\leq \theta\leq 2\pi}|\widetilde{g}_j(\rho_j e^{i\theta})| \geq q_2q_3^2 \cdots 
q_{j-2}^{j-3}q_{j-1}^{j-2}q_j^{j-2}q_{j+1}^{\frac{j-2}{2}} \cdot 
\\ \nonumber &  \bigg(2 - 2q_j\sqrt{q_{j+1}} 
+ q_jq_{j+1} \bigg).
\end{eqnarray}

Second, we estimate the modulus of $\Sigma_1$ from above. We have
\begin{eqnarray}
\label{ll4}
&
|\Sigma_1(\rho_j e^{i \theta})| 
\leq  \sum_{k = 0}^{j-3} \frac{q_2^kq_3^k \cdots
q_j^k q_{j+1}^{\frac{k}{2}}}{q_2^{k-1}q_3^{k-2} \cdots q_k}= 
\\   \nonumber  &(\mbox{we rewrite 
the sum from right} 
 \mbox{  to left})\\   \nonumber &
 = \left( q_2 q_3^2\cdots q_{j-3}^{j-4}q_{j-2}^{j-3}q_{j-1}^{j-3} q_j^{j-3} 
q_{j+1}^{\frac{j-3}{2}} + \right.
\left. q_2 q_3^2 \cdots q_{j-4}^{j-5} q_{j-3}^{j-4}q_{j-2}^{j-4} 
q_{j-1}^{j-4} q_j^{j-4} q_{j+1}^{\frac{j-4}{2}}  \right. \\
 \nonumber & + \left.   q_2 q_3^2 \cdots q_{j-5}^{j-6} q_{j-4}^{j-5}
q_{j-3}^{j-5}q_{j-2}^{j-5}q_{j-1}^{j-5} q_j^{j-5} 
q_{j+1}^{\frac{j-5}{2}} +\cdots    \right) \\
 \nonumber & = q_2 q_3^2 \cdots q_{j-3}^{j-4}q_{j-2}^{j-3}q_{j-1}^{j-3} 
 q_j^{j-3} q_{j+1}^{\frac{j-3}{2}} \\
 \nonumber & \cdot  \left( 1+ \frac{1}{q_{j-2}q_{j-1}q_{j}
 \sqrt{q_{j+1}}}\right. \left.    +  \frac{1}{q_{j-3}q_{j-2}^2 
 q_{j-1}^2q_{j}^2(\sqrt{q_{j+1}})^2} +
 \cdots\right) \\
 \nonumber &  \leq   q_2 q_3^2 \cdots q_{j-3}^{j-4}q_{j-2}^{j-3}
 q_{j-1}^{j-3} q_j^{j-3}q_{j+1}^{\frac{j-3}{2}}\cdot 
 \frac{1}{1- \frac{1}{q_{j-2}q_{j-1}q_{j}\sqrt{q_{j+1}}}}
\end{eqnarray}

 (we estimate the finite sum  from above by the sum of the infinite geometric 
 progression). Finally, we obtain
\begin{eqnarray}
& |\Sigma_1(\rho_j e^{i \theta})|\leq  \\
\nonumber & q_2 q_3^2 \cdots q_{j-3}^{j-4}q_{j-2}^{j-3} 
q_{j-1}^{j-3} q_j^{j-3} q_{j+1}^{\frac{j-3}{2}}\cdot 
 \frac{1}{1- \frac{1}{q_{j-2}q_{j-1}q_{j}\sqrt{q_{j+1}}}} .
\end{eqnarray}

Next, the estimation of $|\Sigma_2(\rho_j e^{i\theta})|$ from above can be made analogously:
\begin{align*}
&  |\Sigma_2(\rho_j e^{i \theta})|
\leq  \sum_{k = j+3}^\infty \frac{q_2^kq_3^k \cdots q_j^k q_{j+1}^{\frac{k}{2}}}{q_2^{k-1}q_3^{k-2} \cdots q_k} 
= \frac{q_2q_3^2 \cdots q_j^{j-1} q_{j+1}^{\frac{j-3}{2}}}{q_{j+2}^2q_{j+3}}\cdot \\
& \bigg( 1 + \frac{1}{\sqrt{q_{j+1}}q_{j+2}q_{j+3}q_{j+4}} + \frac{1}{(\sqrt{q_{j+1}})^2 
q_{j+2}^2q_{j+3}^2q_{j+4}^2q_{j+5}} + \cdots \bigg).
\end{align*}
The latter can be estimated from above by the sum of the geometric progression, so, we obtain
\begin{equation}
|\Sigma_2(\rho_j e^{i \theta})|\leq \frac{q_2q_3^2 \cdots q_j^{j-1} 
q_{j+1}^{\frac{j-3}{2}}}{q_{j+2}^2 q_{j+3}}
\cdot \frac{1}{1 - \frac{1}{\sqrt{q_{j+1}}q_{j+2} q_{j+3} q_{j+4}}}.
\end{equation}
Note that
$$|\xi_j(\rho_j e^{i \theta})| =  q_2q_3^2 \cdots q_{j-2}^{j-3}q_{j-1}^{j-2}q_j^{j-2}
q_{j+1}^{\frac{j-2}{2}} \bigg(1- q_jq_{j+2}^{-1} \bigg).$$

Therefore, the desired inequality $\min_{0\leq \theta\leq 2\pi}|\widetilde{g}_j(\rho_j e^{i\theta})| > 
\max_{0\leq \theta\leq 2\pi}|f (\rho_j e^{i\theta}) -\widetilde{g}_j(\rho_j e^{i\theta})|$ 
follows from
\begin{align*}
& q_2q_3^2 \cdots q_{j-2}^{j-3}q_{j-1}^{j-2}q_j^{j-2}q_{j+1}^{\frac{j-2}{2}} 
\cdot  \bigg(2 - 2q_j\sqrt{q_{j+1}} + q_jq_{j+1} \bigg) \\ 
& > q_2 q_3^2 \cdots q_{j-3}^{j-4}q_{j-2}^{j-3} 
q_{j-1}^{j-3} q_j^{j-3} q_{j+1}^{\frac{j-3}{2}}
\cdot \frac{1}{1- \frac{1}{q_{j-2}q_{j-1}q_{j}\sqrt{q_{j+1}}}} \\
&+ \frac{q_2q_3^2 \cdots q_j^{j-1} 
q_{j+1}^{\frac{j-3}{2}}}{q_{j+2}^2 q_{j+3}}
\cdot \frac{1}{1 - \frac{1}{\sqrt{q_{j+1}}q_{j+2} q_{j+3} q_{j+4}}} \\
&+q_2q_3^2 \cdots q_{j-2}^{j-3}q_{j-1}^{j-2}q_j^{j-2}
q_{j+1}^{\frac{j-2}{2}} \bigg(1 - q_jq_{j+2}^{-1} \bigg).
\end{align*} 
Or, equivalently, 
\begin{eqnarray}
\label{estqq}
& q_{j-1}q_j\sqrt{q_{j+1}} \bigg(2 - 2q_j\sqrt{q_{j+1}} + q_jq_{j+1}\bigg) >
\\  \nonumber &
\frac{1}{1 - \frac{1}{q_{j-2}q_{j-1}q_j\sqrt{q_{j+1}}}} + 
\frac{q_{j-1}q_j^2}{q_{j+2}^2q_{j+3}} \frac{1}{1 - 
\frac{1}{\sqrt{q_{j+1}}q_{j+2}q_{j+3}q_{j+4}}} +
\\  \nonumber & q_{j-1}q_j\sqrt{q_{j+1}}(1 -q_jq_{j+2}^{-1}).
\end{eqnarray}

Since, under our assumptions $q_2 \leq q_3 \leq q_4 \leq \ldots,$
the sequence $(q_j)_{j=2}^\infty$ has the limit, that is finite or infinite.
At first, we consider the case when this limit is finite and put 
$\lim_{j \to \infty} q_j = a,$ $a \geq 2\sqrt[3]{2}.$  We firstly investigate the limiting inequality
\begin{equation}
\label{esta} 
a^2\sqrt{a}(2 - 2a\sqrt{a} + a^2) > \frac{1}{1 - \frac{1}{a^3\sqrt{a}}} + 
\frac{1}{1 - \frac{1}{a^3\sqrt{a}}} + a^2\sqrt{a} \cdot 0.
\end{equation}
Equivalently, 
\begin{align*}
2 - 2a\sqrt{a} + a^2> \frac{2a}{a^3\sqrt{a}-1}.
\end{align*}
Set $\sqrt{a} =: b$, then we obtain  $(2 - 2b^3 + b^4)(b^7 - 1) > 2b^2,$
or
\begin{align*}
b^{11} - 2b^{10} + 2b^7 - b^4 + 2b^3 - 2b^2 - 2 > 0.
\end{align*}

We have found the roots of the polynomial on the left-hand side of the inequality 
using the computer, and its greatest real root is less than $ 1{.}47.$ Thus, 
the inequality is fulfilled for $b > 1{.}47$, and, therefore,  for $a > 2{.}17.$
Under our assumptions, $a \geq 2\sqrt[3]{2} > 2{.}51,$ so the inequality (\ref{esta})
is valid. Whence, for the case when 
the sequence $(q_j)_{j=2}^\infty$ has the finite limit,
 the inequality (\ref{estqq})
is valid  for all $j$ being large enough.

Now we consider the case when $\lim_{j \to \infty} q_j = + \infty.$
The inequality  (\ref{estqq}) follows from the 
\begin{eqnarray}
\label{estqq2}
& q_{j-1}q_j\sqrt{q_{j+1}} \bigg(2 - 2q_j\sqrt{q_{j+1}} + q_jq_{j+1}\bigg) >
\\  \nonumber &
\frac{1}{1 - \frac{1}{q_{j-2}q_{j-1}q_j\sqrt{q_{j+1}}}} + 
 \frac{1}{1 - 
\frac{1}{\sqrt{q_{j+1}}q_{j+2}q_{j+3}q_{j+4}}} +
\\  \nonumber & q_{j-1}q_j\sqrt{q_{j+1}},
\end{eqnarray}
or
$$
2 - 2q_j\sqrt{q_{j+1}} + q_jq_{j+1}>
\frac{1}{q_{j-1}q_j\sqrt{q_{j+1}}}\cdot \frac{1}{1 - 
\frac{1}{q_{j-2}q_{j-1}q_j\sqrt{q_{j+1}}}} + $$
$$\frac{1}{q_{j-1}q_j\sqrt{q_{j+1}}} \cdot \frac{1}{1 - 
\frac{1}{\sqrt{q_{j+1}}q_{j+2}q_{j+3}q_{j+4}}} +1.
$$
The left-hand side of the inequality above tends to
infinity, and the right-hand side tends to $1.$ So, 
the last inequality is valid for all $j$ being large enough.  
Whence, for the case when the sequence $(q_j)_{j=2}^\infty$ 
has the infinite limit,  the inequality (\ref{estqq})
is valid  for all $j$ being large enough.

Consequently, we have proved that for all $j$ being large enough 
$\min_{0\leq \theta\leq 2\pi}|\widetilde{g}_j(\rho_j e^{i\theta})| > 
\max_{0\leq \theta\leq 2\pi} |f(\rho_j e^{i\theta}) -\widetilde{g}_j(\rho_j e^{i\theta})|, $ 
so the number of zeros of 
$f$ in the circle $\{z:|z| < \rho_j\}$ is equal to the number of zeros of 
$\widetilde{g}_j$ in this circle.

In the next stage of the proof, it remains to find the number of zeros 
of $\widetilde{g}_j$ in the circle $\{z:|z| < \rho_j\}$. We have 

$$\widetilde{g}_j(z) = \sum_{k = j - 2}^{j+1} \frac{(-1)^kz^k}{q_2^{k-1} q_3^{k-2} \cdots q_k} + 
\frac{(-1)^{j+2}z^{j+2}}{q_2^{j+1}q_3^j \cdots q_{j-2}^5q_{j-1}^4q_j^4q_{j+1}^2}.$$

Let us use the denotation $w = z \rho_j^{-1},$ so that $|w|<1.$ This yields

$$\widetilde{g}_j(\rho_j w) = (-1)^{j-2}w^{j-2}q_2q_3^2 \cdots 
q_{j-2}^{j-3}q_{j-1}^{j-2}q_j^{j-2}q_{j+1}^{\frac{j-2}{2}} \cdot $$
$$(1 - q_j\sqrt{q_{j+1}}w + q_jq_{j+1}w^2 - q_j\sqrt{q_{j+1}}w^3 + w^4).$$

It follows from (\ref{estg}) that $\widetilde{g}_j$ does not have zeros on the circumference $\{z : |z| = \rho_j  \},$
whence $\widetilde{g}_j(\rho_j w)$ does not have zeros on the circumference $\{w : |w| = 1  \}.$ Since 
$P_j(w) =1-q_j\sqrt{q_{j+1}} w +q_j q_{j+1} w^2 - q_j\sqrt{q_{j+1}} w^3 +w^4 $  is a 
self-reciprocal polynomial in $w,$ we can conclude that $P_j$ has exactly two zeros in the circle $\{ w :|w| <1 \}.$
Hence, $\widetilde{g}_j(z)$ has exactly $j$ zeros in the circle $\{z : |z| < \rho_j  \},$ and we have 
proved the statement of  Lemma~\ref{th:lm3}.
\end{proof}

Theorem~\ref{th:mthm2} is a simple corollary of Lemma~\ref{th:lm3}.

\end{proof}

\section{Proof of Theorem \ref{th:mthm1}}

Without loss of generality, we can assume that $a_0=a_1=1,$ since we can 
consider a function $\psi(z) =a_0^{-1} f (a_0 a_1^{-1}z) $  instead of 
$f(z),$ due to the fact that such rescaling of $f$ preserves its property of 
having real zeros and preserves the second quotients:  $q_n(\psi) =q_n(f)$ 
for all $n.$ For brevity, during the proof we write $p_n$ and $q_n$ instead of 
$p_n(f)$ and $q_n(f).$ Then, we have $$f(z) = 1 + z + \sum_{k=2}^\infty 
\frac{ z^k}{q_2^{k-1} q_3^{k-2} \ldots q_{k-1}^2 q_k}.$$ 

Further, during the proof, we need the inequalities related to the roots of the 
function $f.$ So, for the convenience of dealing with inequalities, we are going to consider the positive roots. Thus, we consider the entire function 

\begin{align*}
\varphi(z) = f(-z) = 1 - z + \sum_{k=2}^\infty  \frac{ (-1)^k z^k}
{q_2^{k-1} q_3^{k-2} \ldots q_{k-1}^2 q_k}
\end{align*}
instead of $f.$

In \cite{ngthv3}, it was proved that if $\varphi \in \mathcal{L-P}$ and $q_2(f) \leq 
q_3(f),$ then there exists $z_0 \in (0; \frac{a_1}{a_2}] = (0, q_2]$ such that 
$\varphi(z_0) \leq 0$ (see Theorem D in the introduction). It remains to prove the 
inverse statement. To do this  we need the following lemma.

\begin{Lemma}
\label{th:lm2} According to Lemma~\ref{th:lm2}, we denote 
by $\rho_k = \rho_k(\varphi) := q_2(\varphi)q_3(\varphi) \cdots q_k(\varphi) 
\sqrt{q_{k+1}(\varphi)},$ 
$k\in\mathbb{N}.$
Under the assumptions of Theorem \ref{th:mthm1}, for every $k\geq 2$  
the following inequality holds: 
$$(-1)^k \varphi(\rho_k) \geq 0.$$
\end{Lemma}

\begin{proof}
Since  $ \rho_k  \in (q_2q_3 \cdots q_k, q_2q_3 \cdots q_kq_{k+1}),$ we have 
\begin{align*}
 1< \rho_k < \frac{\rho_k^2}{q_2} < \cdots < \frac{\rho_k^k}{q_2^{k-1}q_3^{k-2} \cdots q_k},  
 \end{align*}
and
\begin{align*}
 \frac{\rho_k^k}{q_2^{k-1}q_3^{k-2} \cdots q_k} > \frac{\rho_k^{k+1}}{q_2^{k}q_3^{k-1} \cdots q_k^2 q_{k+1}} 
> \frac{\rho_k^{k+2}}{q_2^{k+1}q_3^{k} \cdots q_k^3 q_{k+1}^2 q_{k+2}} > \cdots.
\end{align*}
Therefore, we get for $k\geq 2$
\begin{align*}
(-1)^k\varphi(\rho_k) \geq \sum_{j=k-3}^{k+3}\frac{(-1)^{j+k}\rho_k^j}{q_2^{j-1} q_3^{j-2} \cdots q_j} =: \mu_k(\rho_k),
\end{align*}
and it is sufficient to prove that for every $k\geq 2$ we have $\mu_k(\rho_k)\geq 0.$  
After factoring out $(\rho_k^{k-3})(q_2^{k-4} q_3^{k-5} \cdots q_{k-3})$ the desired inequality takes the form:
\begin{align*}
-1 + \frac{\rho_k}{q_2q_3 \cdots q_{k-3}q_{k-2}} - \frac{\rho_k^2}{q_2^2q_3^2 
\cdots q_{k-2}^2q_{k-1}} + \frac{\rho_k^3}{q_2^3q_3^3 \cdots q_{k-2}^3q_{k-1}^2q_k} \\
-\frac{\rho_k^4}{q_2^4q_3^4 \cdots q_{k-2}^4q_{k-1}^3q_k^2q_{k+1}} + 
\frac{\rho_k^5}{q_2^5q_3^5 \cdots q_{k-2}^5q_{k-1}^4q_k^3q_{k+1}^2q_{k+2}} \\
-\frac{\rho_k^6}{q_2^6q_3^6 \cdots q_{k-2}^6q_{k-1}^5q_k^4q_{k+1}^3q_{k+2}^2q_{k+3}} 
\geq 0,
\end{align*}
or, using that $\rho_k= q_2q_3 \cdots q_k \sqrt{q_{k+1}},$
\begin{align*}
&  \nu_k(\rho_k) := -1 + q_{k-1}q_k\sqrt{q_{k+1}} - 2q_{k-1}q_k^2q_{k+1} + 
q_{k-1}q_k^2q_{k+1}\sqrt{q_{k+1}}\\  &
+ \frac{q_{k-1}q_k^2\sqrt{q_{k+1}}}{q_{k+2}} - 
\frac{q_{k-1}q_k^2}{q_{k+2}^{2}q_{k+3}} \geq 0.
\end{align*}
We observe that
\begin{align*}
& \nu_k(\rho_k) =  (q_{k-1}q_k\sqrt{q_{k+1}} -1) + 
q_{k-1}q_k^2q_{k+1}(\sqrt{q_{k+1}}- 2) + 
\\  &
+\frac{q_{k-1}q_k^2\sqrt{q_{k+1}}}{q_{k+2}} (1 - 
\frac{1}{\sqrt{q_{k+1}} q_{k+2}q_{k+3}}).
\end{align*}

At first, we consider the case when $q_{k+1} \geq 4.$ Then we have
$(q_{k-1}q_k\sqrt{q_{k+1}} -1) >0,$ $q_{k-1}q_k^2q_{k+1}(\sqrt{q_{k+1}}- 2)
\geq 0,$ and $\frac{q_{k-1}q_k^2\sqrt{q_{k+1}}}{q_{k+2}} (1 - 
\frac{1}{\sqrt{q_{k+1}} q_{k+2}q_{k+3}}) >0.$ Thus, in the case
$q_{k+1} \geq 4$ the desired inequality $\nu_k(\rho_k) \geq 0$ is
proved.

Now we consider the case when $q_{k+1} < 4$ and either $q_{k+2} < 4$
(so that $\frac{q_{k}}{q_{k+2}} \geq \frac{3}{4} \geq 0{.}525$),
or  $q_{k+2} \geq 4$  and $\frac{q_{k}}{q_{k+2}} \geq 0{.}525.$  

After rearranging we get 
\begin{align*}
\nu_k(\rho_k) = q_{k-1}q_k^2q_{k+1}\sqrt{q_{k+1}} - 2q_{k-1}q_k^2q_{k+1}\\
+q_{k-1}q_k\sqrt{q_{k+1}} \bigg( 1 + \frac{q_k}{q_{k+2}}\bigg) - \bigg(1 + 
\frac{q_{k-1}q_k^2}{q_{k+2}^2q_{k+3}}\bigg) \geq 0.
\end{align*}

Since $q_k$ are non-decreasing  in $k$, it is easy to see that $(q_{k-1}q_k^2)/(q_{k+2}^2q_{k+3}) 
\leq 1$, hence, it is  sufficient to prove the following inequality
\begin{align*}
q_{k-1}q_k^2q_{k+1}\sqrt{q_{k+1}} - 2q_{k-1}q_k^2q_{k+1}+
q_{k-1}q_k\sqrt{q_{k+1}} \bigg( 1 + \frac{q_k}{q_{k+2}}\bigg) - 2 \geq 0.
\end{align*}

Under our assumptions that $q_k$ are non-decreasing in $k$ and $q_2 \geq 3$, we 
have $2 < \frac{2}{9}q_{k-1}q_k$, and we can observe that
\begin{align*}
q_{k-1}q_k^2q_{k+1}\sqrt{q_{k+1}} &- 2q_{k-1}q_k^2q_{k+1}
+ q_{k-1}q_k\sqrt{q_{k+1}} \cdot 1{.}525 - 2\\ 
&\geq q_{k-1}q_k^2q_{k+1}\sqrt{q_{k+1}} - 2q_{k-1}q_k^2q_{k+1}\\
&+ q_{k-1}q_k\sqrt{q_{k+1}} \cdot 1{.}525 - \frac{2}{9}q_{k-1}q_k.
\end{align*}
So, we need to check that for all $k\geq 2$ 
\begin{align*}
&  q_kq_{k+1}\sqrt{q_{k+1}} - 2q_kq_{k+1} + 1{.}525\sqrt{q_{k+1}} - \frac{2}{9}  =
\\  &
q_kq_{k+1}\left(\sqrt{q_{k+1}} - 2\right) + 1{.}525\sqrt{q_{k+1}} - \frac{2}{9} \geq 0.
\end{align*}
Since $q_k$ is non-decreasing in $k$, we get
\begin{align*}
q_kq_{k+1}\sqrt{q_{k+1}} - 2q_kq_{k+1} + 1{.}525\sqrt{q_{k+1}} - \frac{2}{9} \geq\\
q_{k+1}^2\sqrt{q_{k+1}} - 2q_{k+1}^2 + 1{.}525\sqrt{q_k+1} - \frac{2}{9} .
\end{align*}

Set $\sqrt{q_{k+1}} = t, t \geq 0$, then we obtain the following inequality
\begin{align*}
t^5 - 2t^4 + 1{.}525t - \frac{2}{9} \geq 0.
\end{align*}
This inequality holds for $t \geq 1{.}73051$  (we used numerical methods 
 to find that the greatest real root of the polynomial on the left-hand 
side is less than $ 1{.}73051$), so it follows that it holds
for $q_{k+1} \geq 2{.}99466\ldots$. Thus, in the case
 $q_{k+1} < 4$ and either $q_{k+2} < 4,$
or  $q_{k+2} \geq 4$  and $\frac{q_{k}}{q_{k+2}} \geq 0{.}525$   
the desired inequality $\nu_k(\rho_k) \geq 0$ is
proved.

It remains to consider the case when  $q_{k+1} < 4, $  $q_{k+2} \geq 4,$
and $q_{k+1} \geq 3{.}4303. $ We have 
\begin{eqnarray} \nonumber
& \nu_k(\rho_k) :=  (q_{k-1}q_k\sqrt{q_{k+1}} -1) + 
q_{k-1}q_k^2q_{k+1}(\sqrt{q_{k+1}}- 2) + 
\\  & \nonumber
+\frac{q_{k-1}q_k^2\sqrt{q_{k+1}}}{q_{k+2}} (1 - 
\frac{1}{\sqrt{q_{k+1}} q_{k+2}q_{k+3}}) \geq 
 (q_{k-1}q_k\sqrt{q_{k+1}} -1) + \\   \nonumber &
q_{k-1}q_k^2q_{k+1}(\sqrt{q_{k+1}}- 2) \geq  
(q_{k-1}q_k\sqrt{q_{k+1}} - \frac{q_{k-1}q_k}{9}) +
 \\  \nonumber & q_{k-1}q_k^2q_{k+1}(\sqrt{q_{k+1}}- 2)
= q_{k-1}q_k \left( (\sqrt{q_{k+1}} - \frac{1}{9}) +
\right. 
\\ \nonumber & \left.
q_k q_{k+1} (\sqrt{q_{k+1}}- 2) \right).
\end{eqnarray}
We want to show that 
$$  (\sqrt{q_{k+1}} - \frac{1}{9}) + q_k q_{k+1}(\sqrt{q_{k+1}}- 2) \geq 0.$$
Since $\sqrt{q_{k+1}}- 2 <0,$ and $q_k \leq q_{k+1},$ the last inequality
follows from 
$$  (\sqrt{q_{k+1}} - \frac{1}{9}) + q_{k+1}^2(\sqrt{q_{k+1}}- 2) \geq 0.$$
Denote by $t = \sqrt{q_{k+1}},$ we get the inequality 
$$t^5 -2 t^4 +t -  \frac{1}{9} \geq 0. $$
We have found the roots of the polynomial on the left-hand side of the inequality 
using the computer, and its greatest real root is less than $1{.}8521.$ Thus, 
this inequality valid for $q_{k+1} \geq 3{.}4303. $ So, in the case 
when  $q_{k+1} < 4, $  $q_{k+2} \geq 4,$
and $q_{k+1} \geq 3{.}4303 $ the desired inequality $\nu_k(\rho_k) \geq 0$ is
also proved.

Lemma~\ref{th:lm2} is proved. 
\end{proof}

Suppose that there  exists $z_0 \in (1, q_2),$ such that $\varphi(z_0) \leq 0.$ Then,
by Lemma~\ref{th:lm2}  we have  for every $k\geq 2:$
\begin{align*}
\varphi(0) >0,   \varphi(z_0) \leq 0,  \varphi(\rho_2) \geq 0,  \varphi(\rho_3) \leq 0, \ldots, 
(-1)^k \varphi(\rho_k) \geq 0. 
\end{align*}
So, for every $k\geq 2$ the function $ \varphi$ has at least $k-1$ real zeros on the disk
$\{ z : |z| < \rho_k   \}.$ By Lemma~\ref{th:lm3}, the function $ \varphi$  has exactly
$k$ zeros on the disk $\{ z : |z| < \rho_k \}$ for $k$ being large enough.  So, for all
$k$ being large enough all the zeros of $ \varphi$ on the disk $\{ z : |z| < \rho_k \}$ 
are real.  Thus,  if there  exists $z_0 \in (1, q_2),$ such that $\varphi(z_0) \leq 0,$ then all the
zeros of $\varphi$ are real, thus $\varphi \in \mathcal{L-P}I$.

Theorem~\ref{th:mthm1} is proved. 

{\bf Aknowledgement.} \thanks{ The research was supported  by  the National Research Foundation of Ukraine funded by Ukrainian State budget in frames of project 2020.02/0096 ``Operators in infinite-dimensional spaces:  the interplay between geometry, algebra and topology''.}

\end{document}